\documentclass[fleqn]{article}
\usepackage{amsmath,amsthm,amsfonts}
\usepackage{color}
\usepackage{enumerate}
\usepackage{theoremref}
\usepackage[margin=1in]{geometry}

\numberwithin{equation}{section}
\theoremstyle{plain}
\newtheorem{thm}{Theorem}[section]
\newtheorem{lem}[thm]{Lemma}
\newtheorem{cor}[thm]{Corollary}
\newtheorem{prop}[thm]{Proposition}

\theoremstyle{definition}
\newtheorem{ex}[thm]{Example}

\usepackage[toc,page]{appendix}

\begin{document}
\title{Time Sensitive Analysis of Independent and Stationary Increment Processes}
\author{Jewgeni H. Dshalalow\\
{\color{blue}eugene@fit.edu}
\and
Ryan T. White\\
{\color{blue}rwhite2009@fit.edu}}
\date{}
\maketitle

\vspace{-1cm}
\begin{center}
Department of Mathematical Sciences\\
College of Science\\
Florida Institute of Technology\\
Melbourne, Florida 32901, USA
\end{center}

\begin{abstract}
We study the behavior of independent and stationary increments jump processes as they approach fixed thresholds. The exact crossing time is unavailable because the real-time information about successive jumps is unknown. Instead, the underlying process $A(t)$ is observed only upon a third-party independent point process $\{\tau_n\}$. The observed time series $\{A(\tau_n)\}$ presents crude, delayed data. The crossing is first observed upon one of the observations, denoted $\tau_\nu$. We develop and further explore a new technique to revive the real-time paths of $A(t)$ for all $t$ belonging to an interval before the pre-crossing observation, $\left[0,\tau_{\nu-1}\right)$, or between the observations just before and just after the crossing, $\left[\tau_{\nu-1},\tau_\nu\right)$, as a joint Laplace-Stieltjes transform and probability generating function of $A(\tau_{\nu-1})$, $A(\tau_\nu)$, $\tau_{\nu-1}$, and $\tau_\nu$. Joint probability distributions are obtained from the transforms in a tractable form and they are applied to modeling of stochastic networks under cyber attacks by accurately predicting their crash.


\end{abstract}

\section{Introduction}

Among various problems related to fluctuations of independent and stationary increments (ISI) processes, there are many where an underlying process attempts to escape from a fixed set, and cannot be observed precisely at the moment of crossing, but with some delay. The first time upon which an observer can ascertain the crossing will be referred to as a first observed crossing (or passage) time. Such a delay is inevitable due to the limited observational capabilities allowed by a third-party point process, thus obscuring information about the first passage time and the location of the process upon the first passage time.

In recent research by the authors \cite{DshalalowWhite13,DshalalowWhite14}, efforts were made to derive probabilistic information about the location of such processes and times upon the observation immediately before and after the crossing, as well as locations of the process and times at various other observation epochs. The authors derived joint Laplace-Stieltjes transforms (LST) and probability generating functions (PGF) of these positions and times, which led to marginal distributions and moments for each of these random variables (r.v.'s).

In the present paper we use an entirely different and seemingly more efficient approach to refine predicted data that offers potential for stochastic control. To be more specific, suppose we collect the location of a process $A(t)$ at times $\{\tau_1,\tau_2,...\}$ and let $\nu$ be a random index when the crossing is first observed: that is, $\tau_\nu$ is the first observed passage time. One of the intervals of interest is $\left[\tau_{\nu-1},\tau_\nu\right)$ a random vicinity of time before and after the first passage time, in which
we can glean some information on $A(t)$ and obtain the joint distribution of the process at any given time $t\in\left[\tau_{\nu-1},\tau_\nu\right)$ and at $\tau_\nu$ (along with other notable characteristics), which proved elusive under the previous strategy.

This paper further focuses on a Poisson point process $\{t_k\}\in\mathbb{R}_+$ with associated position-independent marking $\{a_k\}\subseteq\mathbb{N}$ taking the role of $A(t)$, which is observed upon a delayed renewal process ${\{\tau_n, n=0,1,2,...\}}$ forming a sequence $A_n=A(\tau_n)$, which is an embedded marked ISI process. With ${\nu=\inf\{n:A_n>M\}}$ (where $M$ is a fixed positive integer), $\tau_\nu$ is the first observed crossing time and $A_\nu$ is the position of $A(t)$ upon this time.

We study the following two joint LST's and PGF's of the position of the process $A(t)$ at times $t$, $\tau_{\nu-1}$, and $\tau_\nu$ as well as the latter two times themselves. The two functionals are restricted to time intervals before the pre-crossing time, $\left[0,\tau_{\nu-1}\right)$, and between observations before and after the crossing, $\left[\tau_{\nu-1},\tau_\nu\right)$,

\begin{align}
G_1(t,u,v,w,x,y) &= E\left[u^{A_{\nu-1}}v^{A_\nu}e^{-w\tau_{\nu-1}-x\tau_\nu}y^{A(t)}\mathbf{1}_{\{t<\tau_{\nu-1}\}}\right]\label{G1}
\\ G_2(t,u,v,w,x,y) &= E\left[u^{A_{\nu-1}}v^{A_\nu}e^{-w\tau_{\nu-1}-x\tau_\nu}y^{A(t)}\mathbf{1}_{\{\tau_{\nu-1}\leq t<\tau_\nu\}}\right].\label{G2}
\end{align}

To address this problem in a more general fashion we start with a generic ISI process $A(t)$ and two independent random variables $T$ and $\Delta$, which are also independent of $A(t)$. The latter condition is necessary to derive explicit functionals, but $\tau_{\nu-1}$ and $\tau_\nu$ (assuming the roles of $T$ and $T+\Delta$) are neither mutually independent nor independent of $A(t)$. For that reason, we ``expand'' the above functionals in stochastic series over $\tau_n$'s so that each term of the series satisfies the required independence assumptions. The series converge to useful expressions for \eqref{G1} and \eqref{G2}. Herewith we arrive at closed form expressions in terms of the Laplace transforms of $G_1$ and $G_2$. The total tractability of these expressions is illustrated via special cases.

Applications of ISI processes under delayed observation are widespread. The authors previously investigated strategic networks under random attacks in which random quantities of nodes (along with associated weights) get disabled, modeled as an ISI process in 2D. The well-being of the network is monitored by delayed observations that offer a more realistic scenario than real-time information. If a certain number of nodes or weights become paralyzed, the entire network is disabled. In order to prevent such an event, we predicted the time at which a cumulative damage process exits from a rectangle and its excess value. With the new, time-sensitive, techniques, we largely improve our earlier (``time-insensitive'') models \cite{DshalalowWhite13,DshalalowWhite14}. The present work establishes a more general and systematic approach of stochastic interpolation that can naturally lead to further embellishments, e.g. with $A(t)$ being multivariate and/or crossing multiple thresholds.

We note that our work falls into the class of the theory of fluctuations of stochastic processes that historically stem from the random walk behavior of a particle moving throughout a discrete deterministic grid and attempting to escape from a fixed set as in \cite{Bingham2001}. Later on, the notion of fluctuations was utilized to other stochastic processes like renewal, recurrent, and semi-Markov that were to cross a fixed threshold. (See Tak\'{a}cs' seminal paper \cite{Takacs1978}.) This terminology, as well as the utility of a process leaving a fixed set, was found in various work ranging from insurance, risk, and reliability \cite{Borovkov08} to finance \cite{Dshalalow:2005bh} to astronomy, biology, physics, and quantum physics \cite{Hida1995,Redner2001}, and to networks \cite{DshalalowWhite13,DshalalowWhite14}. Topically, papers on fluctuations of marked Poisson and Cox processes \cite{Agarwal:2004fk}, renewal and recurrent processes \cite{Takacs1978}, and ISI processes \cite{Kadankov2005} are more pertinent to our present work. 

The time-sensitive tools can be utilized in other time-insensitive models, in particular in stochastic games, queueing, and finance \cite{Dshalalow:2005bh,Dshalalow2009}. In some earlier forms, time-sensitive analysis was known in queueing \cite{DshalalowAlMatar11} (under the name of semi-regenerative techniques) and in other stochastic processes such as Cox \cite{Agarwal:2004fk} to interpolate stationary probabilities of embedded processes.

The layout of the article is as follows. In Section~\ref{GeneralResultsSection}, we state and prove some auxiliary results for general ISI processes, which are of separate interest. Namely, we consider an ISI process $A(t)$, observed at random times $T$ and $T+\Delta$ where $T$ and $\Delta$ are independent r.v.'s, which are also independent of $A(t)$. We derive expressions for functionals similar to \eqref{G1} and \eqref{G2}, but based on times $T$ and $T+\Delta$ rather than $\tau_{\nu-1}$ and $\tau_\nu$. As already mentioned, the results of Section~\ref{GeneralResultsSection} do not directly apply if $T$ and $T+\Delta$ are $\tau_{\nu-1}$ and $\tau_\nu$ because the latter are dependent in general. To overcome this obstacle, in Section~\ref{DelayedPassageSection} we reduce the functionals of Section~\ref{GeneralResultsSection} to the case when $A(t)$ is a marked Poisson process observed by a renewal process and use a series expansion technique to arrive at Laplace transforms of functionals of the type \eqref{G1} and \eqref{G2}, which we then illustrate explicitly in Section~\ref{SpecialCaseResultsSection} by examples.

\section{Some Important Results for General ISI Processes}\label{GeneralResultsSection}

Suppose $A(t)$ is an independent and stationary increment (ISI) stochastic process valued in $\{0,1,2,...\}$ defined on a filtered probability space $\left(\Omega,\mathcal{F},\left(\mathcal{F}_t\right)_{t\geq 0},P\right)$. Suppose $T$ and $\Delta$ are nonnegative, independent r.v.'s, which are independent of the filtration $\left(\mathcal{F}_t\right)$. We will seek functionals
\begin{align}
F_1(t,u,v,w,x,y) &= E\left[u^{A(T)}v^{A(T+\Delta)}e^{-wT-x\Delta}y^{A(t)}\mathbf{1}_{\{t<T\}}\right]\label{F1}
\\ F_2(t,u,v,w,x,y) &= E\left[u^{A(T)}v^{A(T+\Delta)}e^{-wT-x\Delta}y^{A(t)}\mathbf{1}_{\{T\leq t<T+\Delta\}}\right]\label{F2}
\end{align}
where $t\geq 0$ and
\begin{align*}
u,v,y\in\left\{z\in\mathbb{C}:\|z\|\leq 1\right\}, w,x\in\left\{z\in\mathbb{C}:Re(z)\geq 0\right\}
\end{align*}
Note \eqref{F1}-\eqref{F2} are somewhat similar to \eqref{G1}-\eqref{G2} with $T$ and $\Delta$ replacing $\tau_{\nu-1}$ and $\tau_\nu-\tau_{\nu-1}$. The key dissimilarity between the two is that $\tau_{\nu-1}$ and $\tau_\nu$ are not independent of the filtration $\left(\mathcal{F}_t\right)$.

For an $L^1$-function or functional $f(t)$, denote its Laplace transform by $f^*(\theta)=\mathcal{L}_t\{f(t)\}(\theta)$.

\begin{thm}\thlabel{FirstTheorem}
The functional $F_1(t,u,v,w,x,y)$ of the ISI process $A(t)$ on the trace $\sigma$-algebra ${\mathcal{F}\cap\{t<T\}}$ where $T$ and $\Delta$ are independent of each other and of the filtration $(\mathcal{F}_t)$ satisfies
\begin{align}
F_1^*(\theta,u,v,w,x,y)=E\left[e^{-wT}\psi(uvy,uv,T)\right]E\left[e^{-x\Delta}\varphi(v,\Delta)\right]\label{F1star}
\end{align}
where
\begin{align}
\varphi(a,s)=E\left[a^{A(s)}\right]\label{AsIncrement}
\end{align}
is the PGF of the process at time $s$, and
\begin{align}
\psi(b,c,\delta)=\left(e^{-\theta(\cdot)}\varphi(b,\cdot)\right)*\varphi(c,\cdot)(\delta)=\int_0^\delta \varphi(b,t)\varphi(c,\delta-t)\,dt\label{psi}
\end{align}
is the convolution of $e^{-\theta(\cdot)}\varphi(b,\cdot)$ with $\varphi(c,\cdot)$ valued at $\delta$.
\end{thm}
\begin{proof}
First, rewrite
\begin{align}
&A(T) = A(T) - A(t) + A(t),\,t\in[0,T)\label{SplitAT}
\\&A(T+\Delta) = A(T+\Delta) - A(T) + A(T) - A(t) + A(t),\,t\in[0,T)\label{SplitATplusDelta}
\end{align}
Then, using \eqref{SplitAT}-\eqref{SplitATplusDelta}, the assumed independence of $(\mathcal{F}_t)$ from the $\sigma$-algebra $\sigma(T,\Delta)$ (induced by $T$ and $\Delta$), and the ISI property of $A(t)$, we have
\begin{align*}
&E\left[u^{A(T)}v^{A(T+\Delta)}e^{-wT-x\Delta}y^{A(t)}\mathbf{1}_{\{t<T\}}\Big|\sigma(T,\Delta)\right],
\\&=\mathbf{1}_{\{t<T\}}e^{-wT-x\Delta}E\left[(uv)^{A(T) - A(t)}v^{A(T+\Delta)-A(T)}(uvy)^{A(t)}\Big|\sigma(T,\Delta)\right]
\\&=\mathbf{1}_{\{t<T\}}e^{-wT-x\Delta}E\left[(uv)^{A(T-t)}\Big|\sigma(T)\right]E\left[
v^{A(\Delta)}\Big|\sigma(\Delta)\right]E\left[(uvy)^{A(t)}\right].\label{FirstTheorem1}
\end{align*}
By notation \eqref{AsIncrement}-\eqref{psi},
\begin{align}
E\left[u^{A(T)}v^{A(T+\Delta)}e^{-wT-x\Delta}y^{A(t)}\mathbf{1}_{\{t<T\}}\Big|\sigma(T,\Delta)\right]=\mathbf{1}_{\{t<T\}}e^{-wT-x\Delta}\varphi(uv,T-t)\varphi(v,\Delta)\varphi(uvy,t).
\end{align}
Next, we seek
\begin{align*}
F_1^*(\theta,u,v,w,x,y) = \int_{t\geq 0} e^{-\theta t}F_1(t,u,v,w,x,y)\,dt.
\end{align*}
Rewriting $F_1$ in the integral form and using \eqref{FirstTheorem1}, we have
\begin{align*}
&F_1^*(\theta,u,v,w,x,y)
\\&=\int_{t\geq 0} e^{-\theta t}\int_{(s,\delta)\in\mathbb{R}^2_+} e^{-ws-x\delta}\varphi(uvy,t)\varphi(uv,s-t)\varphi(v,\delta)\mathbf{1}_{\{t<s\}}\,dP_{T\otimes\Delta}(s,\delta)\,dt.
\end{align*}
By the independence of $T$ and $\Delta$ and Fubini's Theorem,
\begin{align*}
F_1^*(\theta,u,v,w,x,y)&= \int_{\delta\geq 0} e^{-x\delta}\varphi(v,\delta)\,dP_\Delta(\delta)\,\int_{s\geq 0} e^{-ws}\int_{t=0}^s e^{-\theta t}\varphi(uvy,t)\varphi(uv,s-t)\,dt\,dP_T(s)
\\&=E\left[e^{-x\Delta}\varphi(v,\Delta)\right]\int_{s\geq 0} \psi(uvy,uv,s)\,dP_T(s)
\\&=E\left[e^{-x\Delta}\varphi(v,\Delta)\right]E\left[e^{-wT}\psi(uvy,uv,T)\right].
\end{align*}
\end{proof}

\begin{thm}\thlabel{SecondTheorem}
The functional $F_2(t,u,v,w,x,y)$ of the ISI process $A(t)$ on the trace $\sigma$-algebra

\vspace{-.3cm}
${\mathcal{F}\cap\{T\leq t<T+\Delta\}}$ where $T$ and $\Delta$ are independent of each other and of the process $A(t)$ satisfies
\begin{align}
F_2^*(\theta,u,v,w,x,y)=E\left[e^{-(w+x)T}\varphi(uvy,T)\right]E\left[e^{-x\Delta}\psi(vy,v,\Delta)\right],
\end{align}
where $\varphi$ and $\psi$ are defined in \eqref{AsIncrement}-\eqref{psi}.
\end{thm}
\begin{proof}
First, notice
\begin{align}
&A(t) = A(t) - A(T) + A(T)\label{SecondProof1}
\\&A(T+\Delta) = A(T+\Delta) - A(t) + A(t) - A(T) + A(T)\label{SecondProof2}
\end{align}
for $t\in[T,T+\Delta)$. Then, using \eqref{SecondProof1}-\eqref{SecondProof2}, the ISI property of $A(t)$, and that the r.v.'s $T$ and $\Delta$ are $\sigma(T,\Delta)$-measurable, we have
\begin{align}
&E\left[u^{A(T)}v^{A(T+\Delta)}e^{-wT-x\Delta}y^{A(t)}\mathbf{1}_{\{T\leq t<T+\Delta\}}\Big|\sigma(T,\Delta)\right]\notag
\\&=e^{-wT-x\Delta}\mathbf{1}_{\{T\leq t<T+\Delta\}}E\left[(uvy)^{A(T)}(vy)^{A(t)-A(T)}u^{A(T+\Delta)-A(t)}\Big|\sigma(T,\Delta)\right]\notag
\\&=e^{-wT-x\Delta}\mathbf{1}_{\{T\leq t<T+\Delta\}}\varphi(uvy,T)\varphi(vy,t-T)\varphi(v,T+\Delta-t),\label{SecondProof3}
\end{align}
where $\varphi(u,s)=E\left[u^{A(s)}\right]$.

Next, we seek
\begin{align*}
F_2^*(\theta,u,v,w,x,y)=\int_{t\geq 0}e^{-\theta t}F_2(t,u,v,w,x,y)\,dt.
\end{align*}
Writing $F_2$ in the integral form and using \eqref{SecondProof3}, we have
\begin{align*}
&F_2^*(\theta,u,v,w,x,y)
\\&=\int_{t\geq 0}e^{-\theta t}\int_{(s,\delta)\in\mathbb{R}^2_+} e^{-ws-x\delta}\varphi(uvy,s)\varphi(vy,t-s)\varphi(v,s+\delta-t)\mathbf{1}_{\{s\leq t<s+\delta\}}\,dP_{T\otimes\Delta}(s,\delta)\,dt.
\end{align*}
By the independence of $T$ and $\Delta$ and Fubini's Theorem,
\begin{align*}
&F_2^*(\theta,u,v,w,x,y)
\\&= \int_{s\geq 0}e^{-(w+\theta)s}\varphi(uvy,s)\,dP_T(s)\int_{\delta\geq 0}e^{-x\delta}\int_{t=s}^{s+\delta}e^{-\theta(t-s)}\varphi(vy,t-s)\varphi(v,s+\delta-t)\,dt\,dP_\Delta(\delta).
\end{align*}
Using the ISI property and the change of variables $z=t-s$,
\begin{align*}
F_2^*(\theta,u,v,w,x,y)&= E\left[e^{-(w+\theta)T}\varphi(uvy,T)\right]\int_{\delta\geq 0}e^{-x\delta}\int_{z=0}^\delta e^{-\theta z}\varphi(vy,z)\varphi(v,\delta-z)\,dz\,dP_\Delta(\delta)
\\&=E\left[e^{-(w+\theta)T}\varphi(uvy,T)\right]\int_{\delta\geq 0}e^{-x\delta}\psi(vy,v,\delta)\,dP_\Delta(\delta)
\\&=  E\left[e^{-(w+\theta)T}\varphi(uvy,T)\right]E\left[e^{-x\Delta}\psi(vy,v,\Delta)\right].
\end{align*}
\end{proof}

\begin{ex}
We demonstrate the applicability of \thref{FirstTheorem,SecondTheorem} by considering the special case of $A(t)$ as a marked (compound) Poisson process, with position-independent marking, i.e.
\begin{align}
A(t) = \sum\limits_{k\geq 0} a_k\varepsilon_{t_k}[0,t],
\end{align}
where $\varepsilon_c$ is the Dirac point mass centered at $c\in\mathbb{R}_+$ and $a_k:\Omega\rightarrow\{0,1,2,...\}$ are independent and identically distributed (IID) r.v.'s with common PGF $g(z)=E\left[e^{a_1}\right]$, $\|z\|\leq 1$, and are independent of $\{t_k\}$. Furthermore, the underlying Poisson counting measure $\sum_{k=0}^\infty \varepsilon_{t_k}$ is of intensity $\lambda$, so
\begin{align}
E\left[z^{A(t)}\right] = e^{\lambda t[g(z)-1]}.
\label{ThirdProof1}
\end{align}
Note that the marked Poisson process $A(t)$ is a real-time process evolving over a sequence of epochs $\{t_k\}$. In Section~\ref{DelayedPassageSection}, we will consider a ``third-party" renewal process (mentioned in the introduction) over which $A(t)$ will be observed, forming the time series $\{A(\tau_n)\}$.
\end{ex}

Let $L_X(z) = E\left[e^{-zX}\right]$ denote the LST of a random variable $X$.
\begin{lem}\thlabel{ThirdTheorem}
Let $A(t)$ be a marked Poisson process of intensity $\lambda$ and let $T$ and $\Delta$ be r.v.'s independent of each other and of the filtration $(\mathcal{F}_t)$, then
\begin{align}
&F_1^*(\theta,u,v,w,x,y)=\frac{L_T(w+\lambda-\lambda g(uv)) - L_T(\theta+w+\lambda-\lambda g(uvy))}{\theta + \lambda g(uv) - \lambda g(uvy)}L_\Delta(x+\lambda-\lambda g(v))
\\&F_2^*(\theta,u,v,w,x,y)=L_T(\theta+w+\lambda-\lambda g(uvy))\frac{L_\Delta(x+\lambda-\lambda g(v)) - L_\Delta(\theta+x+\lambda-\lambda g(vy))}{\theta + \lambda g(v) - \lambda g(vy)},
\end{align}
where $g(z)=E\left[z^{a_1}\right]$ is the common PGF of the marks of $a_k$'s, the r.v. marks of the Poisson process.
\begin{proof}
Since $A(t)$ is a marked Poisson process, \eqref{ThirdProof1} implies
\begin{align}
&\varphi(a,X)=E\left[e^{-aA(X)}\Big|X\right] = e^{-\lambda(1-g(a))X}\label{ThirdProof2}
\end{align}
and
\begin{align}
\psi(b,c,X) &= \int_0^X e^{-\theta t}\varphi(b,t)\varphi(c,X-t)\,dt\notag
\\&= \int_0^X e^{-\theta t}e^{-\lambda(1-g(b))t}e^{-\lambda(1-g(c))(X-t)}\,dt\notag
\\&=e^{-\lambda(1-g(c))X}\int_0^X e^{-(\theta+\lambda g(c) - \lambda g(b))t}\,dt\notag
\\&=\frac{1}{\theta+\lambda g(c) -\lambda g(b)}\left(e^{-\lambda(1-g(c))X}-e^{-(\theta+\lambda-\lambda g(b))X}\right).\label{ThirdProof3}
\end{align}
The Poissonian assumption leads to expressions \eqref{ThirdProof2}-\eqref{ThirdProof3}, which simplify the result of \thref{FirstTheorem} to
\begin{align*}
F_1^*(\theta,u,v,w,x,y)&=\frac{E\left[e^{-wT}\left(e^{-\lambda(1-g(uv))T}-e^{-(\theta + \lambda - \lambda g(uvy))T}\right)\right]}{\theta + \lambda g(uv) - \lambda g(uvy)} E\left[e^{-x\Delta}e^{-\lambda(1-g(v))\Delta}\right]
\\&= \frac{L_T(w + \lambda - \lambda g(uv))-L_T(\theta+w+\lambda-\lambda g(uvy))}{\theta + \lambda g(uv) - \lambda g(uvy)}L_\Delta(x+\lambda-\lambda g(v)).
\end{align*}
Similarly, by \thref{SecondTheorem},
\begin{align*}
F_2^*(\theta,u,v,w,x,y)&=E\left[e^{-(w+x)T}e^{-\lambda(1-g(uvy))T}\right]\frac{E\left[e^{-x\Delta}\left(e^{-\lambda(1-g(v))\Delta}-e^{-(\theta+\lambda-\lambda g(vy))\Delta}\right)\right]}{\theta+\lambda g(v)-\lambda g(vy)}
\\&=L_T(\theta+w+\lambda-\lambda g(uvy))\frac{L_\Delta(x+\lambda-\lambda g(v)) - L_\Delta(\theta+x+\lambda-\lambda g(vy))}{\theta + \lambda g(v) - \lambda g(vy)}.
\end{align*}
\end{proof}
\end{lem}

\section{Delayed Passage Time Model}\label{DelayedPassageSection}

As noted in Section~\ref{GeneralResultsSection}, $A(t)$ is a marked Poisson process where the marks of the process, ${a_k:\Omega\rightarrow\{0,1,2,...\}}$, are IID r.v.'s with common PGF $g(z)=E\left[z^{a_1}\right]$, $\|z\|\leq 1$.

Next, we define a delayed renewal process $\mathcal{T}=\{\tau_n:n=0,1,...\}$ over which $A(t)$ will be observed. We assume that $\tau_n:\Omega\rightarrow\mathbb{R}_+$, $n=0,1,...$, are independent of the filtration $(\mathcal{F}_t)$. As mentioned in the introduction, real-time information about the process may not be available to the observer. Denote the inter-observation times as $\Delta_n=\tau_n-\tau_{n-1}$ for $n\geq 0$ and $\tau_{-1}=0$.

We assume the initial observation time $\tau_0=\Delta_0$ has a Laplace-Stieltjes transform (LST) ${L_0(\theta)=E\left[e^{-\theta\Delta_0}\right]}$ and all subsequent inter-observation times have a common LST of ${L(\theta)=E\left[e^{-\theta\Delta_1}\right]}$, $Re(\theta)\geq 0$.

Denote the observed value of the process up to time $\tau_n$ (the $n$th observation) by
\begin{align}
A_n=A(\tau_n),n\geq 0,
\end{align}
and the $n$th increment between the successive observations, as
\begin{align}
X_n=A(\tau_n)-A(\tau_{n-1})=A_n-A_{n-1},
\end{align}
where $A_{-1}=0$.

Since the very first observation epoch of time $\Delta_0$ need not be stochastically equivalent to $\Delta_n$ for $n\geq 1$, we will use a separate notation for the joint transform of the increment up to $\Delta_0$ and the time $\Delta_0$ itself,
\begin{align}
\gamma_0(z,\theta)=E\left[z^{X_0}e^{-\theta\Delta_0}\right],
\end{align}
and the notation for the $n$th increment and inter-observation time for all other $n$,
\begin{align}
\gamma(z,\theta)=E\left[z^{X_1}e^{-\theta\Delta_1}\right],\label{JointTransformOfIncrement}
\end{align}
with $\|z\|\leq 1$ and $Re(\theta)\geq 0$.

We will be interested in the distribution of the first observation upon which the process $A(t)$ exceeds a fixed threshold $M\in\mathbb{N}$, so we define the \textit{exit index} as a random integer corresponding to the first observation after a threshold crossing,
\begin{align}
\nu=\inf\{n=0,1,2,...:A_n>M\}
\end{align}
while $\tau_\nu$ will be called the \textit{first observed crossing time}. The goal of this section will be, under the condition that $A(t)$ is a marked Poisson process, to derive the joint functionals,
\begin{align}
&G_1^*(\theta,u,v,w,x,y)=\int_{t\geq 0}e^{-\theta t} E\left[u^{A_{\nu-1}}v^{A_\nu}e^{-w\tau_{\nu-1}-x\Delta_\nu}y^{A(t)}\mathbf{1}_{\{t<\tau_{\nu-1}\}}\right]\,dt\label{G1Star}
\\&G_2^*(\theta,u,v,w,x,y)=\int_{t\geq 0}e^{-\theta t}E\left[u^{A_{\nu-1}}v^{A_\nu}e^{-w\tau_{\nu-1}-x\Delta_\nu}y^{A(t)}\mathbf{1}_{\{\tau_{\nu-1}\leq t<\tau_\nu}\right]\,dt,\label{G2Star}
\end{align}
of the values of $A$ upon the observations immediately before and after the real-time crossing ($A_{\nu-1}$ and $A_\nu$), the real-time value of the process $A(t)$, and the times of the observations immediately before and after the crossing ($\tau_{\nu-1}$ and $\tau_\nu$). Notice that each functional deals with $t$ placed within a particular random time interval, either before $\tau_{\nu-1}$ or between $\tau_{\nu-1}$ and $\tau_\nu$. These are the Laplace transforms of functionals \eqref{G1} and \eqref{G2}, which are the primary goal of the study. Notice that in contrast to the assumptions of \thref{FirstTheorem,SecondTheorem}, $\tau_{\nu-1}$ and $\tau_\nu$ (that seem to resemble r.v.'s $T$ and $T+\Delta$) are no longer independent of the filtration $(\mathcal{F}_t)$. To make use of these theorems, we will use a special stochastic decomposition of the exit index $\nu$.

Thus, consider the family of exit indices
\begin{align}
\nu(p)=\inf\{n=0,1,2,...:A_n>p\}
\end{align}
corresponding to the first observed crossing of a threshold $p$. For a bounded function $f:\mathbb{N}\rightarrow\mathbb{R}_+$, we will further use the operator
\begin{align}
\text{D}_p(f(p))(s) = (1-s)\sum\limits_{p=0}^\infty s^p f(p), s\in B(0,1)\subseteq\mathbb{C},
\end{align}
where $B(0,1)$ is the open unit ball centered at zero in the complex plane. [Note that the subscript $p$ is being used in the event an underlying function has other parameters or variables.]

The operator $\mathcal{D}^k$ is applied to a complex-valued function $\varphi(s)$, analytic at 0,
\begin{align}
\mathcal{D}_s^k(\varphi(s))= \left\{
     \begin{array}{lr}
       \lim\limits_{s\rightarrow 0}\frac{1}{k!}\frac{\partial^k}{\partial s^k}\left[\frac{1}{1-s}\varphi(s)\right], & \text{if } k\geq 0\\
       0, & \text{if } k<0
     \end{array}
   \right.\label{DInverse}
\end{align}
Obviously $\{\mathcal{D}^k:k=0,1,...\}$ is the inverse of D$_p$. That is, $\mathcal{D}_s^k\left[\text{D}_p(f(p))(s)\right] = f(k)$, $k=0,1,...$

In the following theorem, we will apply the transform D$_p$ to a sequence of functionals including a factor of the form $\mathbf{1}_{\{t<\tau_{\nu(p)-1}\}}$ inside the expectation, i.e. a restriction of the functional to the set where time is before the observation just prior to the first observed crossing of a threshold $p$. In Section~\ref{SpecialCaseResultsSection}, we will demonstrate how to apply the inverse operator $\mathcal{D}_s^{M-1}$ to restore the functional where $p=M$.

Denote
\begin{align}
&G_{1,j}^*(\theta,u,v,w,x,y)=\int_{t\geq 0}e^{-\theta t}E\left[u^{A_{j-1}}v^{A_j}e^{-w\tau_{j-1}-x\Delta_j}y^{A(t)}\mathbf{1}_{\{t<\tau_{j-1}\}}\right]\,dt\label{G1jStar}
\\&G_{2,j}^*(\theta,u,v,w,x,y)=\int_{t\geq 0}e^{-\theta t}E\left[u^{A_{j-1}}v^{A_j}e^{-w\tau_{j-1}-x\Delta_j}y^{A(t)}\mathbf{1}_{\{\tau_{j-1}\leq t<\tau_j\}}\right]\,dt.\label{G2jStar}
\end{align}

Note that these functionals are similar to \eqref{G1Star}-\eqref{G2Star}, except they specify a fixed index $j$ indicating the r.v.'s in the functionals are associated specifically with the $j$th observation epoch. The utility of these functionals will be seen in the proofs of \thref{FourthTheorem,FifthTheorem}, as \eqref{G1Star} and \eqref{G2Star} will be expanded into series over this index $j$ of functionals of types \eqref{G1jStar} and \eqref{G2jStar} with the additional factor $\mathbf{1}_{\{\nu(p)=j\}}$ fixing $\nu(p)=j$ within the expectations.

\begin{thm}\thlabel{FourthTheorem}
The joint functional $G_1^*(\theta,u,v,w,x,y)$ of the process $A(t)$ on the interval $[0,\tau_{\nu-1})$ satisfies
\begin{align}
G_1^*(\theta,u,v,w,x,y) = \mathcal{D}_s^{M-1}\left\{B_1(B_2-B_3)\right\}
\end{align}
where
\begin{align}
B_1 = \frac{\gamma(v,x)-\gamma(vs,x)}{\theta + \lambda g(uvs) - \lambda g(uvys)},\hspace{1cm} B_2=\frac{\gamma_0(uvs,w)}{1-\gamma(uvs,w)},\hspace{1cm} B_3=\frac{\gamma_0(uvys,\theta+w)}{1-\gamma(uvys,\theta+w)}.\label{FourthProof2}
\end{align}
\end{thm}
\begin{proof}
Denote
\begin{align*}
\Phi_1(\theta,s)=\text{D}_p\left(\int_{t\geq 0} e^{-\theta t} E\left[u^{A_{\nu(p)-1}}v^{A_{\nu(p)}}e^{-w\tau_{\nu(p)-1}-x\Delta_{\nu(p)}}y^{A(t)}\mathbf{1}_{\{t<\tau_{\nu(p)-1}\}}\right]\,dt\right)(s)
\end{align*}
and notice
\begin{align*}
&u^{A_{\nu(p)-1}}v^{A_\nu(p)}e^{-w\tau_{\nu(p)-1}-x\Delta_\nu}y^{A(t)}\mathbf{1}_{\{t<\tau_{\nu(p)-1}\}}
\\&=\sum\limits_{j\geq 0} u^{A_{j-1}}v^{A_j}e^{-w\tau_{j-1}-x\Delta_j}y^{A(t)}\mathbf{1}_{\{\nu(p)=j\}}\mathbf{1}_{\{t<\tau_{j-1}\}}.
\end{align*}
Then
\begin{align}
\Phi_1(\theta,s)=\text{D}_p\left(\int_{t\geq 0} e^{-\theta t} E\left[\sum\limits_{j\geq 0} u^{A_{j-1}}v^{A_j}e^{-w\tau_{j-1}-x\Delta_j}y^{A(t)}\mathbf{1}_{\{\nu(p)=j\}}\mathbf{1}_{\{t<\tau_{j-1}\}}\right]\,dt\right)(s).
\end{align}
Notice
\begin{align}
\text{D}_p\left(\mathbf{1}_{\nu(p)=j}\right)(s)=(1-s)\sum\limits_{p=A_{j-1}}^{A_j-1} s^p = s^{A_{j-1}}-s^{A_j}.
\end{align}
Then, by Fubini's Theorem and since $\mathbf{1}_{\{t<\tau_{j-1}\}}=0$ for $j=0$,
\begin{align*}
\Phi_1(\theta,s)&=\sum\limits_{j>0}\int_{t\geq 0} e^{-\theta t} E\left[u^{A_{j-1}}v^{A_j}e^{-w\tau_{j-1}-x\Delta_j}y^{A(t)}\left(s^{A_{j-1}}-s^{A_j}\right)\mathbf{1}_{\{t<\tau_{j-1}\}}\right]\,dt
\\&=\sum\limits_{j>0} G_{1,j}^*(\theta,us,v,w,x,y)-G_{1,j}^*(\theta,u,vs,w,x,y).
\end{align*}
Because the observation process $\mathcal{T}=\{\tau_n\}$ is independent of $A(t)$, \thref{ThirdTheorem} implies that, for $j=1,2,...$,
\begin{align}
&G_{1,j}^*(\theta,us,v,w,x,y)\notag
\\&=\left[\gamma_0(uvs,w)\gamma^{j-1}(uvs,w)-\gamma_0(uvys,\theta+w)\gamma^{j-1}(uvys,\theta+w)\right]\frac{\gamma(vs,x)}{\theta+\lambda g(uvs)-\lambda g(uvys)}\label{FourthProof1}
\end{align}
Notice that the norm of the joint transform \eqref{JointTransformOfIncrement} of the increment of the process upon the $n$th ($n\geq 1$) inter-observation time is strictly less than 1 in each underlying instance. This is due to the argument in Appendix~\ref{AppendixB} (i.e. ${\|\gamma(uvs,w)\|<1}$ and ${\|\gamma(uvys,\theta+w)\|<1}$). Therefore, summing up \eqref{FourthProof1} for ${j=1,2,...}$ yields a geometric series,
\begin{align*}
&\Phi_1(\theta,s)
\\&=\frac{\gamma(v,x)-\gamma(vs,x)}{\theta+\lambda g(uvs)-\lambda g(uvys)}\left(\sum\limits_{j>0}\left[\gamma_0(uvs)\gamma^{j-1}(uvs,w)-\gamma_0(uvys,\theta+w)\gamma^{j-1}(uvys,\theta+w)\right]\right)
\\&=\frac{\gamma(v,x)-\gamma(vs,x)}{\theta+\lambda g(uvs)-\lambda g(uvys)}\left(\frac{\gamma_0(uvs,w)}{1-\gamma(uvs,w)}-\frac{\gamma_0(uvys,\theta+w)}{1-\gamma(uvys,\theta+w)}\right) = B_1(B_2-B_3),
\end{align*}
where $B_1$, $B_2$, and $B_3$ are defined in \eqref{FourthProof2}.
\end{proof}

\begin{thm}\thlabel{FifthTheorem}
The joint functional $G_2^*(\theta,u,v,w,x,y)$ of the process $A(t)$ on the interval $[\tau_{\nu-1},\tau_\nu)$ satisfies
\begin{align}
G_2^*(\theta,u,v,w,x,y)=\mathcal{D}_s^{M-1}\left(\Gamma_0+\Gamma B_3\right)
\end{align}
where
\begin{align}
&\Gamma_0 = \frac{\gamma_0(v,x)-\gamma_0(vy,\theta+x)}{\theta+\lambda g(v)-\lambda g(vy)}-\frac{\gamma_0(vs,x)-\gamma_0(vys,\theta+x)}{\theta+\lambda g(vs)-\lambda g(vys)},
\\&\Gamma = \frac{\gamma(v,x)-\gamma(vy,\theta+x)}{\theta+\lambda g(v)-\lambda g(vy)}-\frac{\gamma(vs,x)-\gamma(vys,\theta+x)}{\theta + \lambda g(vs) - \lambda g(vys)},
\end{align}
and $B_3$ is defined in \eqref{FourthProof2}.
\end{thm}
\begin{proof}
Denote
\begin{align*}
\Phi_2(\theta,s)&=\mathcal{D}_p\left(\int_{t\geq 0} e^{-\theta t}E\left[u^{A_{\nu(p)-1}}v^{A_{\nu(p)}}e^{-w\tau_{\nu(p)-1}-x\Delta_{\nu(p)}}y^{A(t)}\mathbf{1}_{\{\tau_{\nu(p)-1}\leq t<\tau_{\nu(p)}\}}\right]\,dt\right)(s)
\\&=\mathcal{D}_p\left(\int_{t\geq 0} e^{-\theta t}E\left[\sum\limits_{j\geq 0} u^{A_{j-1}}v^{A_j}e^{-w\tau_{j-1}-x\Delta_j} y^{A(t)}\mathbf{1}_{\{\nu(p)=j\}}\mathbf{1}_{\{\tau_{j-1}\leq t<\tau_j\}}\right]\,dt\right)(s).
\end{align*}
Notice
\begin{align*}
&u^{A_{\nu(p)-1}}v^{A_{\nu(p)}}e^{-w\tau_{\nu(p)-1}-x\Delta_{\nu(p)}}y^{A(t)}\mathbf{1}_{\{\tau_{\nu(p)-1}\leq t<\tau_{\nu(p)}\}}
\\&=\sum\limits_{j\geq 0} u^{A_{j-1}}v^{A_j}e^{-w\tau_{j-1}-x\Delta_j} y^{A(t)}\mathbf{1}_{\{\nu(p)=j\}}\mathbf{1}_{\{\tau_{j-1}\leq t<\tau_j\}}.
\end{align*}
Then, by Fubini's Theorem,
\begin{align}
\Phi_2(\theta,s)&=\int_{t\geq 0}e^{-\theta t} E\left[v^{A_0}e^{-x\Delta_0}\left(1-s^{A_0}\right)y^{A(t)}\mathbf{1}_{\{0\leq t<\tau_0\}}\right]\,dt\notag
\\&\hspace{1cm}\times\sum\limits_{j>0}e^{-\theta t}\int_{t\geq 0}E\left[u^{A_{j-1}}v^{A_j}e^{-w\tau_{j-1}-x\Delta_j}y^{A(t)}\left(s^{A_{j-1}}-s^{A_j}\right)\mathbf{1}_{\{\tau_{j-1}\leq t<\tau_j\}}\right]\,dt\notag
\\&=G_{2,0}^*(t,1,v,0,x,y)-G_{2,0}^*(t,1,vs,0,x,y)\notag
\\&\hspace{1cm}+\sum\limits_{j>0} G_{2,j}^*(t,us,v,w,x,y)-G_{2,j}^*(t,u,vs,w,x,y).\label{FifthProof1}
\end{align}
Since the observation process is independent of $A(t)$, \thref{ThirdTheorem} implies
\begin{align}
&G_{2,0}^*(t,1,v,0,x,y)-G_{2,0}^*(t,1,vs,0,x,y)\notag
\\&=\frac{\gamma_0(v,x)-\gamma_0(vy,\theta+x)}{\theta + \lambda g(v) - \lambda g(vy)}-\frac{\gamma_0(vs,x)-\gamma_0(vys,\theta+x)}{\theta + \lambda g(vs) - \lambda g(vys)}=\Gamma_0\label{FifthProof2}
\end{align}
and, for $j=1,2,...$,
\begin{align}
&G_{2,j}^*(t,us,v,w,x,y)=\gamma_0(uvys,\theta+w)\gamma^{j-1}(uvys,\theta+w)\frac{\gamma(v,x)-\gamma(vy,\theta+x)}{\theta + \lambda g(v) - \lambda g(vy)}
\\&G_{2,j}^*(t,u,vs,w,x,y)=\gamma_0(uvys,\theta+w)\gamma^{j-1}(uvys,\theta+w)\frac{\gamma(vs,x)-\gamma(vys,\theta+x)}{\theta + \lambda g(vs) - \lambda g(vys)}.\label{FifthProof3}
\end{align}
Summing \eqref{FifthProof1} via \eqref{FifthProof2}-\eqref{FifthProof3} over $j=1,2,...$ yields a geometric series,
\begin{align*}
\Phi_2(\theta,s)=\Gamma_0+\Gamma\gamma_0(uvys,\theta+w)\sum\limits_{j>0}\gamma^{j-1}(uvys,\theta + w)=\Gamma_0+\Gamma B_3.
\end{align*}
(See a pertinent argument in the paragraph below \eqref{FourthProof1}.)
\end{proof}

\begin{cor}\thlabel{Corollary}
The joint functional $G^*(\theta,u,v,w,x,y)$ of the process $A(t)$ on the interval $[0,\tau_\nu)$ satisfies
\begin{align}
G^*(\theta,u,v,w,x,y)&=\int_{t\geq 0}e^{-\theta t}E\left[u^{A_{\nu-1}}v^{A_\nu}e^{-w\tau_{\nu-1}-x\Delta_\nu}y^{A(t)}\mathbf{1}_{\{t<\tau_{\nu}\}}\right]\,dt\notag
\\&= G_1^*(\theta,u,v,w,x,y) + G_2^*(\theta,u,v,w,x,y)\notag
\\&= \mathcal{D}_s^{M-1}\left\{B_1(B_2 - B_3)+\Gamma_0+\Gamma B_3\right\}.
\end{align}
\end{cor}

\section{Results for a Special Case}\label{SpecialCaseResultsSection}

Note that the results of the prior sections consider the marks and inter-observation times to be simply IID, with no further assumptions on the distributions. The purpose of this section is to demonstrate that the results of the prior sections can readily lead to expressions suitable for numerical and analytical manipulation. As such, we will now derive some results for the one-dimensional marked Poisson process discussed in the previous section under the following assumptions:
\begin{enumerate}
\item The marks $a_k$ are geometrically distributed with parameter $a$, so their common PGF is ${g(z)=\frac{az}{1-bz}}$ for $k=1,2,...$ (where $b = 1-a$)
\item The inter-observation times $\Delta_k$ are exponentially distributed with parameter $\mu$, so their common LST is ${L(z)=\frac{\mu}{\mu+z}}$ for $k=1,2,...$
\item We set the initial observation $\tau_0=0$ and the initial value of the process $A(\tau_0)=0$. Thus, we have ${\gamma_0(\cdot,\cdot)=1}$.
\end{enumerate}

While it is not significantly more complex to find $G_1^*(\theta,u,v,w,x,y)$ and $G_2^*(\theta,u,v,w,x,y)$ in full generality, in the interest of brevity, we will restrict our attention to the following demonstration of the utility of the results above by deriving the joint probability distribution $P\{A_\nu=r,\tau_{\nu-1}>t\}$, i.e. the distribution of the position of the process upon the first observed crossing, and the time of the observation just prior to the real-time crossing.

The first step is to use assumptions 1-3 to simplify the result of \thref{FourthTheorem}, which allows us to apply the inverse operator $\mathcal{D}_s^{M-1}$ explicitly to find $G_1^*(\theta,1,v,0,0,1)$.

\begin{prop}\thlabel{SixthTheorem}
Under assumptions 1-3, the Laplace transform of the probability generating function of the process upon the first observed passage time, $A_\nu$, restricted to ${\{t<\tau_{\nu-1}\}}$, is
\begin{align}
G_1^*(\theta,1,v,0,0,1)&=\int_{t\geq 0} e^{-\theta t}E\left[v^{A_\nu}\mathbf{1}_{\{t<\tau_{\nu-1}\}}\right]\,dt\notag
\\&=\frac{1}{\theta}\left\{\gamma(v,0)\frac{\mu+\lambda}{\lambda}\left[v^{M-1}+(1-F(\mu,v))\sum\limits_{j=0}^{M-2} v^j\right]\right.\notag
\\&\hspace{1cm}\left. -\gamma(v,0)\frac{\mu+\theta+\lambda}{\theta+\lambda}\left[F(\theta,v)^{M-1}+(1-F(\mu+\theta,v))\sum\limits_{j=0}^{M-2} F(\theta,v)^j\right]\right\}\notag
\\&\hspace{.4cm}-\frac{\mu}{\lambda}\left[\sum\limits_{j=0}^{M-1} v^j \sum\limits_{i=0}^{M-1-j} F(\mu,v)^i - (F(\mu,v)+bv)\sum\limits_{j=0}^{M-2}v^j\sum\limits_{i=0}^{M-2-j} F(\mu,v)^i\right.\notag
\\&\hspace{1.4cm} \left. + bv F(\mu,v)\sum\limits_{j=0}^{M-3} v^j\sum\limits_{i=0}^{M-3-j} F(\mu,v)^i\right]\notag
\\&\hspace{.4cm}+\frac{\mu}{\mu+\lambda}\frac{\mu+\theta+\lambda}{\theta+\lambda}\left[\sum\limits_{j=0}^{M-1} F(\theta,v)^j\sum\limits_{i=0}^{M-1-j} F(\mu,v)^i\right.\notag
\\&\hspace{3.6cm} - (F(\mu+\theta,v)+bv)\sum\limits_{j=0}^{M-2}F(\theta,v)^j\sum\limits_{i=0}^{M-2-j} F(\mu,v)^i\notag
\\&\hspace{3.6cm} \left. + bvF(\mu+\theta,v)\sum\limits_{j=0}^{M-3} F(\theta,v)^j\sum\limits_{i=0}^{M-3-j} F(\mu,v)^i\right],\label{SixthProof4}
\end{align}
where $F(x,v)=\frac{bx+\lambda}{x+\lambda}v$.
\end{prop}
\begin{proof}
By \thref{FourthTheorem} and the linearity of the inverse $\mathcal{D}_s^{M-1}$ (since the limit and derivative shown in \eqref{DInverse} are linear),
\begin{align}
&G_1^*(\theta,1,v,0,0,1)=\int_{t\geq 0}e^{-\theta t}E\left[v^{A_\nu}\mathbf{1}_{\{t<\tau_{\nu-1}\}}\right]\,dt\notag
\\&=\frac{1}{\theta}\mathcal{D}_s^{M-1}\left\{(\gamma(v,0)-\gamma(vs,0))\left(\frac{1}{1-\gamma(vs,0)}-\frac{1}{1-\gamma(vs,\theta)}\right)\right\}\notag
\\&=\frac{1}{\theta}\left[\gamma(v,0)\mathcal{D}_s^{M-1}\left\{\frac{1}{1-\gamma(vs,0)}\right\}-\mathcal{D}_s^{M-1}\left\{\frac{\gamma(vs,0)}{1-\gamma(vs,0)}\right\}\right.\notag
\\&\hspace{1cm}\left. -\gamma(v,0)\mathcal{D}_s^{M-1}\left\{\frac{1}{1-\gamma(vs,\theta)}\right\}+\mathcal{D}_s^{M-1}\left\{\frac{\gamma(vs,0)}{1-\gamma(vs,\theta)}\right\}\right].\label{SixthProof1}
\end{align}
Using the result of Appendix~\ref{AppendixA} and assumptions 1-3, we first manipulate some terms into more convenient forms,
\begin{align*}
&\frac{1}{1-\gamma(vs,\theta)}=\frac{1}{1-\frac{\mu}{\mu+\theta+\lambda-\lambda g(vs)}}=\frac{\mu+\theta+\lambda-\lambda\frac{avs}{1-bvs}}{\theta+\lambda-\lambda\frac{avs}{1-bvs}}=\frac{\mu+\theta+\lambda}{\theta+\lambda}\frac{1-F(\mu+\theta,v)s}{1-F(\theta,v)s}
\\&\gamma(vs,0)=\frac{\mu}{\mu+\lambda-\lambda g(vs)}=\frac{\mu}{\mu+\lambda}\frac{1-bvs}{1-F(\mu,v)s}.
\end{align*}
The term ${\mathcal{D}_s^{M-1}\left\{\frac{1}{1-\gamma(vs,\theta)}\right\}}$ from the last line of \eqref{SixthProof1} simplifies as follows,
\begin{align}
&\mathcal{D}_s^{M-1}\left\{\frac{1}{1-\gamma(vs,\theta)}\right\}=\frac{\mu+\theta+\lambda}{\theta+\lambda}\mathcal{D}_s^{M-1}\left\{\frac{1-F(\mu+\theta,v)s}{1-F(\theta,v)s}\right\}\notag
\\&=\frac{\mu+\theta+\lambda}{\theta+\lambda}\left[\mathcal{D}_s^{M-1}\left\{\frac{1}{1-F(\theta,v)s}\right\}-F(\mu+\theta,v)\mathcal{D}_s^{M-2}\left\{\frac{1}{1-F(\theta,v)s}\right\}\right]\notag
\\&=\frac{\mu+\theta+\lambda}{\theta+\lambda}\left[\sum\limits_{j=0}^{M-1} F(\theta,v)^j-F(\mu+\theta,v)\sum\limits_{j=0}^{M-2}F(\theta,v)^j\right]\notag
\\&=\frac{\mu+\theta+\lambda}{\theta+\lambda}\left[F(\theta,v)^{M-1}+(1-F(\mu+\theta,v))\sum\limits_{j=0}^{M-2}F(\theta,v)^j\right].\label{SixthProof2}
\end{align}
The term $\mathcal{D}_s^{M-1}\left\{\frac{\gamma(vs,0)}{1-\gamma(vs,\theta)}\right\}$ from the last line of \eqref{SixthProof1} simplifies as follows,
\begin{align}
&\mathcal{D}_s^{M-1}\left\{\frac{\gamma(vs,0)}{1-\gamma(vs,\theta)}\right\}=\frac{\mu}{\mu+\lambda}\frac{\mu+\theta+\lambda}{\theta+\lambda}\mathcal{D}_s^{M-1}\left\{\frac{1-bvs}{1-F(\mu,v)s}\frac{1-F(\mu+\theta,v)s}{1-F(\theta,v)s}\right\}\notag
\\&=\frac{\mu}{\mu+\lambda}\frac{\mu+\theta+\lambda}{\theta+\lambda}\mathcal{D}_s^{M-1}\left\{\frac{1-(F(\mu+\theta,v)+bv)s+bvF(\mu+\theta,v)s^2}{(1-F(\mu,v)s)(1-F(\theta,v)s)}\right\}\notag
\\&=\frac{\mu}{\mu+\lambda}\frac{\mu+\theta+\lambda}{\theta+\lambda}\left[\mathcal{D}_s^{M-1}\left\{\frac{1}{(1-F(\mu,v)s)(1-F(\theta,v)s)}\right\}\right.\notag
\\&\hspace{3.2cm}-(F(\mu+\theta,v)+bv)\mathcal{D}_s^{M-2}\left\{\frac{1}{(1-F(\mu,v)s)(1-F(\theta,v)s)}\right\}\notag
\\&\hspace{3.2cm}\left. +bvF(\mu+\theta,v)\mathcal{D}_s^{M-3}\left\{\frac{1}{(1-F(\mu,v)s)(1-F(\theta,v)s)}\right\}\right]\notag
\\&=\frac{\mu}{\mu+\lambda}\frac{\mu+\theta+\lambda}{\theta+\lambda}\left[\sum\limits_{j=0}^{M-1}F(\theta,v)^j\sum\limits_{i=0}^{M-1-j}F(\mu,v)^i\right.\notag
\\&\hspace{3.2cm}-(F(\mu+\theta,v)+bv)\sum\limits_{j=0}^{M-2}F(\theta,v)^j\sum\limits_{i=0}^{M-2-j}F(\mu,v)^i\notag
\\&\hspace{3.2cm}\left. +bvF(\mu+\theta,v)\sum\limits_{j=0}^{M-3}F(\theta,v)^j\sum\limits_{i=0}^{M-3-j}F(\mu,v)^i\right].\label{SixthProof3}
\end{align}
The two terms $\mathcal{D}_s^{M-1}\left\{\frac{1}{1-\gamma(vs,0)}\right\}$ and $\mathcal{D}_s^{M-1}\left\{\frac{\gamma(vs,0)}{1-\gamma(vs,\theta)}\right\}$ are special cases of \eqref{SixthProof2} and \eqref{SixthProof3}, respectively, and follow trivially with $\theta=0$. Plugging expressions \eqref{SixthProof2}-\eqref{SixthProof3} and those that follow for the other two terms into \eqref{SixthProof1} yields \eqref{SixthProof4}, as required.
\end{proof}

The next step is to invert the Laplace transform to find ${E\left[v^{A_\nu}\mathbf{1}_{\{t<\tau_{\nu-1}\}}\right]}$, which appears similar to the PGF of $A_\nu$ except for the inclusion an of indicator function restricting it to the set $\{t<\tau_{\nu-1}\}$ as a factor within the expectation.

\begin{prop}\thlabel{SeventhTheorem}
Under assumptions 1-3, the expectation $E\left[v^{A_\nu}\mathbf{1}_{\{t<\tau_{\nu-1}\}}\right]$ satisfies
\begin{align}
&E\left[v^{A_\nu}\mathbf{1}_{\{t<\tau_{\nu-1}\}}\right]\notag
\\&=\gamma(v,0)\frac{\mu+\lambda}{\lambda}\left[v^{M-1}+(1-F(\mu,v))\sum\limits_{j=0}^{M-2}v^j\right]-\gamma(v,0)\left[v^{M-1}G_{M-1}+\sum\limits_{j=0}^{M-2}(v^jG_j+v^{j+1}H_j)\right]\notag
\\&\hspace{.5cm}-\frac{\mu}{\lambda}\left[\sum\limits_{j=0}^{M-1}v^j\sum\limits_{i=0}^{M-1-j}F(\mu,v)^i - (F(\mu,v)+bv)\sum\limits_{j=0}^{M-2}v^j\sum\limits_{i=0}^{M-2-j}F(\mu,v)^i\right.\notag
\\&\hspace{1.5cm}\left.+bvF(\mu,v)\sum\limits_{j=0}^{M-3}v^j\sum\limits_{i=0}^{M-3-j} F(\mu,v)^i\right]\notag
\\&\hspace{.5cm}+\frac{\mu}{\mu+\lambda}\left[\sum\limits_{j=0}^{M-1}v^jG_j\sum\limits_{i=0}^{M-1-j}F(\mu,v)^i-\sum\limits_{j=0}^{M-2}v^{j+1}(bG_j+H_j)\sum\limits_{i=0}^{M-2-j}F(\mu,v)^i\right.\notag
\\&\hspace{1.7cm}\left.+b\sum\limits_{j=0}^{M-3}v^{j+2}H_j\sum\limits_{i=0}^{M-3-j}F(\mu,v)^i\right]\label{SeventhProof5}
\end{align}
where 
\begin{align}
&G_j=b^j\sum\limits_{k=0}^j\binom{j}{k}\left(\frac{a}{b}\right)^k\left[P(k,\lambda t)+\frac{\mu}{\lambda}P(k+1,\lambda t)\right]
\\&H_j=b^{j+1}\sum\limits_{k=0}^j\binom{j}{k}\left(\frac{a}{b}\right)^k\left[P(k,\lambda t)+\left(\frac{\mu}{\lambda}+\frac{a}{b}\right)P(k+1,\lambda t)\right]
\end{align}
and $P(k,\lambda t)=1-\frac{\Gamma(k,\lambda)}{\Gamma(k)}$ is the lower regularized gamma function.
\begin{proof}
The required result can be found by applying the inverse Laplace transform with respect to $\theta$ to the result of \thref{SixthTheorem}, i.e. $\mathcal{L}_\theta^{-1}\left\{G_1^*(\theta,1,v,0,0,1)\right\}(t)$. By linearity, it can be seen that all non-trivial Laplace inverses are taken of terms of the following two forms,
\begin{align}
&\frac{1}{\theta}\frac{\mu+\theta+\lambda}{\theta+\lambda}F(\theta,v)^j=\frac{1}{\theta}\frac{\mu+\theta+\lambda}{\theta+\lambda}\left(\frac{b\theta+\lambda}{\theta+\lambda}\right)^j v^j,\,j=0,1,...,M-1\label{SeventhProof1}
\\&\frac{1}{\theta}\frac{\mu+\theta+\lambda}{\theta+\lambda}F(\mu+\theta,v)F(\theta,v)^j=\frac{1}{\theta}\frac{b\theta+b\mu+\lambda}{\theta+\lambda}\left(\frac{b\theta+\lambda}{\theta+\lambda}\right)^j v^{j+1},\,j=0,1,...,M-1.\label{SeventhProof2}
\end{align}
The inverse Laplace transform of \eqref{SeventhProof1} is
\begin{align}
&v^j\mathcal{L}_\theta^{-1}\left\{\frac{1}{\theta}\frac{\mu+\theta+\lambda}{\theta+\lambda}\left(\frac{b\theta+\lambda}{\theta+\lambda}\right)^j\right\}(t)=v^j e^{-\lambda t}\mathcal{L}_\theta^{-1}\left\{\frac{1}{\theta-\lambda}\frac{\mu+\theta}{\theta}\left(\frac{a\lambda+b\theta}{\theta}\right)^j\right\}(t)\notag
\\&=(bv)^je^{-\lambda t}\sum\limits_{k=0}^j\binom{j}{k}\left(\frac{a\lambda}{b}\right)^k\mathcal{L}_\theta^{-1}\left\{\frac{\theta+\mu}{\theta^{k+1}(\theta-\lambda)}\right\}(t)\notag
\\&=(bv)^je^{-\lambda t}\sum\limits_{k=0}^j\binom{j}{k}\left(\frac{a\lambda}{b}\right)^k\frac{e^{\lambda t}}{\lambda^k}\left[P(k,\lambda t)+\frac{\mu}{\lambda}P(k+1,\lambda t)\right]\notag
\\&=(bv)^j\sum\limits_{k=0}^j\binom{j}{k}\left(\frac{a}{b}\right)^k\left[P(k,\lambda t)+\frac{\mu}{\lambda}P(k+1,\lambda t)\right]=v^jG_j.\label{SeventhProof3}
\end{align}
The inverse Laplace transform of \eqref{SeventhProof2} is
\begin{align}
&v^{j+1}\mathcal{L}_\theta^{-1}\left\{\frac{1}{\theta}\frac{\mu+\theta+\lambda}{\theta+\lambda}\left(\frac{b\theta+b\mu+\lambda}{\theta+\mu+\lambda}\right)\left(\frac{b\theta+\lambda}{\theta+\lambda}\right)^j\right\}(t)\notag
\\&=v^{j+1}e^{-\lambda t}\mathcal{L}_\theta^{-1}\left\{\frac{b\theta+b\mu+a\lambda}{\theta+\mu+\lambda}\frac{(b\theta+a\lambda)^j}{\theta^{j+1}}\right\}(t)\notag
\\&=(bv)^{j+1}e^{-\lambda t}\sum\limits_{k=0}^j\binom{j}{k}\left(\frac{a\lambda}{b}\right)^k\left[\mathcal{L}_\theta^{-1}\left\{\frac{1}{\theta^k(\theta-\lambda)}\right\}(t)+\left(\mu+\frac{a\lambda}{b}\right)\mathcal{L}_\theta^{-1}\left\{\frac{1}{\theta^{k+1}(\theta-\lambda)}\right\}(t)\right]\notag
\\&=(bv)^{j+1}e^{-\lambda t}\sum\limits_{k=0}^j\binom{j}{k}\left(\frac{a}{b}\right)^k\left[P(k,\lambda t)+\left(\frac{\mu}{\lambda}+\frac{a}{b}\right)P(k+1,\lambda t)\right]=v^{j+1}H_j.\label{SeventhProof4}
\end{align}
Therefore, applying the inverse Laplace transform to \eqref{SixthProof4} reduces it to a linear combination of results of the forms \eqref{SeventhProof3} and \eqref{SeventhProof4}, yielding the result of the proposition, \eqref{SeventhProof5}.
\end{proof}
\end{prop}

From this result, one further step will yield the explicit joint distribution of $A_\nu$ and $\tau_{\nu-1}$.

\begin{prop}\thlabel{EighthTheorem}
The joint distribution of the pre-first observed passage time, $\tau_{\nu-1}$, and the crossing level of the process, $A_\nu$, is
\begin{align}
&P\left\{A_\nu=r,\tau_{\nu-1}>t\right\}\notag
\\&=\frac{\mu}{\lambda}\frac{a\mu}{\mu+\lambda}\left[\frac{\mu+\lambda}{a\mu}R_{0r}+\sum\limits_{j=1}^{M-1}R_{jr}\right]-\frac{\mu}{\mu+\lambda}\left[G_0R_{0r}+\sum\limits_{j=1}^{M-1}(G_j-H_{j-1})R_{jr}\right]\notag
\\&\hspace{.5cm}-\frac{\mu}{\lambda}\left[\sum\limits_{j=0}^{M-1}\sum\limits_{i=0}^{M-1-j}c^i\mathbf{1}_{\{r=i+j\}}-(b+c)\sum\limits_{j=0}^{M-2}\sum\limits_{i=0}^{M-2-j}c^i\mathbf{1}_{\{r=i+j+1\}}+bc\sum\limits_{j=0}^{M-3}\sum\limits_{i=0}^{M-3-j}c^i\mathbf{1}_{\{r=1+j+2\}}\right]\notag
\\&\hspace{.5cm}+\frac{\mu}{\mu+\lambda}\left[\sum\limits_{j=0}^{M-1}G_j\sum\limits_{i=0}^{M-1-j}c^i\mathbf{1}_{\{r=i+j\}}-\sum\limits_{j=0}^{M-2}(bG_j+H_j)\sum\limits_{i=0}^{M-2-j}c^i\mathbf{1}_{\{r=i+j+1\}}\right.\notag
\\&\hspace{2.1cm}\left.+b\sum\limits_{j=0}^{M-3}H_j\sum\limits_{i=0}^{M-3-j}c^i\mathbf{1}_{\{r=1+j+2\}}\right]\label{EighthProof3}
\end{align}
where
\begin{align}
&c=F(\mu,1)=\frac{b\mu+\lambda}{\mu+\lambda}
\\&R_{jr}=\left\{
     \begin{array}{lr}
       0, & \text{if } r<j\\
       1, & \text{if } r=j\\
       (c-b)c^{r-j-1}, & \text{if } r>j
     \end{array}
   \right.\label{EighthProof1}
\end{align}
\end{prop}
\begin{proof}
Since the result of \thref{SeventhTheorem} is a PGF of $A_\nu$ (restricted to $\{t<\tau_{\nu-1}\}$), we can find the joint distribution as
\begin{align*}
P\{A_\nu=r,\tau_{\nu-1}>t\}=\frac{1}{r!}\lim\limits_{v\rightarrow 0}\frac{\partial^r}{\partial v^r}E\left[v^{A_\nu}\mathbf{1}_{\{t<\tau_{\nu-1}\}}\right].
\end{align*}
We have
\begin{align*}
\gamma(v,0)=\frac{\mu}{\mu+\lambda-\lambda g(v)}=\frac{\mu}{\mu+\lambda}\frac{1-bv}{1-\frac{b\mu+\lambda}{\mu+\lambda}v}=\frac{\mu}{\mu+\lambda}\frac{1-bv}{1-cv}.
\end{align*}
As such, by the linearity of derivatives and limits, all $v$-dependent terms fall into the form $\frac{1-bv}{1-cv}v^j$ or $v^j$, so we find (noting $\|cu\|<1$)
\begin{align*}
R_{jr}&=\frac{1}{r!}\lim\limits_{v\rightarrow 0}\frac{
\partial^r}{\partial v^r}\left(\frac{1-bv}{1-cv}v^r\right)=\frac{1}{r!}\lim\limits_{v\rightarrow 0}\frac{
\partial^r}{\partial v^r}\left(v^j(1-bv)\sum\limits_{i=0}^\infty (cv)^i\right)
\\&=\sum\limits_{i=0}^\infty c^i\frac{1}{r!}\lim\limits_{v\rightarrow 0}\frac{\partial^r}{\partial v^r}\left(v^{i+j}\right)-b\sum\limits_{i=0}^\infty c^i\frac{1}{r!}\lim\limits_{v\rightarrow 0}\frac{\partial^r}{\partial v^r}\left(v^{i+j+1}\right)
\\&=\left\{
     \begin{array}{lr}
       0, & \text{if } r<j\\
       1, & \text{if } r=j\\
       (c-b)c^{r-j-1}, & \text{if } r>j
     \end{array}
   \right.
\end{align*}
and
\begin{align}
\frac{1}{r!}\lim\limits_{v\rightarrow 0}\frac{\partial^r}{\partial v^r}\left(v^j\right)=\mathbf{1}_{\{j=r\}}.\label{EighthProof2}
\end{align}
Therefore, taking the appropriate limit and derivatives of the PGF \eqref{SeventhProof5} of \thref{SeventhTheorem} yields a linear combination of terms of the forms \eqref{EighthProof1} and \eqref{EighthProof2}, yielding \eqref{EighthProof3} as required.
\end{proof}

The process of deriving the joint distribution consisted of (a) simplifying \thref{FourthTheorem} to the marginal transform $G_1^*(\theta,1,v,0,0,1)=\int_{t\geq 0}e^{-\theta t}E\left[v^{A_\nu}\mathbf{1}_{\{t<\tau_{\nu-1}\}}\right]\,dt$ under assumptions 1-3 above (\thref{SixthTheorem}), (b) taking the inverse Laplace transform to find $E\left[v^{A_\nu}\mathbf{1}_{\{t<\tau_{\nu-1}\}}\right]$ (\thref{SeventhTheorem}), and (c) finding the joint probability distribution $P\{A_\nu=r,\tau_{\nu-1}>t\}$ (\thref{EighthTheorem}). Any such joint distribution of $A_{\nu-1}$, $A(t)$, or $A_\nu$ with $\tau_{\nu-1}$ or $\tau_\nu$ (via \thref{Corollary}) can be derived through a procedure with the same three general steps.

\section*{Acknowledgement}

The authors are very grateful to the anonymous referee whose numerous remarks and suggestions greatly improved the mathematical contents of the paper, as well as its overall readability.

\bibliographystyle{plain}
\bibliography{TSAISIP}

\appendix
\begin{appendices}
\section{Simplification of $\gamma(z,\theta)$ for a Marked Poisson Process}\label{AppendixA}

\begin{proof}
If the underlying process $A(t)$ is a marked Poisson process, the joint PGF and LST of an observed increment $X_1$ and the inter-observation time $\Delta_1$ will simply as follows by the known property of the conditional expectation,
\begin{align*}
\gamma(z,\theta) = E\left[z^{X_1}e^{-\theta\Delta_1}\right] = E\left[e^{-\theta\Delta_1}E\left[z^{X_1}\big|\Delta_1\right]\right] = E\left[e^{-\theta\Delta_1} E\left[z^{a_1+...+a_N}\big|\Delta_1\right]\right],
\end{align*}
where $N$ is the number of Poisson arrivals occurring within a time interval of length $\Delta_1$. Since the Poisson process has rate $\lambda$, $N$ is a Poisson random variable with parameter $\lambda\Delta_1$, then we have
\begin{align*}
\gamma(z,\theta) = E\left[e^{-\theta\Delta_1}e^{-\left(\lambda g(z) - \lambda\right)\Delta_1}\right] = E\left[e^{-\left(\theta + \lambda - \lambda g(z)\right)\Delta_1}\right] = L\left(\theta + \lambda - \lambda g(z)\right).
\end{align*}
\end{proof}

\section{Proof that $\|L\left(\theta + \lambda + \lambda g(z)\right)\|<1$}\label{AppendixB}

Herein, we will show the norm of the joint PGF and LST of an observed increment $X_1$ and the inter-observation time $\Delta_1$, ${\gamma(z,\theta)=E\left[z^{X_1}e^{-\theta\Delta_1}\right]=L(\theta+\lambda-\lambda g(z))}$ (as shown in Appendix~\ref{AppendixA}) is strictly bounded by 1 so that the geometric series will converge in the proofs above under a minor condition imposed on the parameters.

Recall that by definition, $\lambda>0$, $Re(\theta)\geq 0$ (it could be zero if we seek the marginal PGF of $X_1$), and $\|z\|\leq 1$ (it could be 1 if we seek the marginal LST of $\Delta_1$). The convergence will occur if $\|z\|<1$ or $Re(\theta)>0$ (i.e. if either inequality is strict), as we show below.

\begin{proof}
Let $\eta=\theta+\lambda-\lambda g(z)$, then
\begin{align*}
\|\gamma(z,\theta)\|&=\|L(\theta+\lambda-\lambda g(z))\|=\|L(\eta)\left\|=\|E\left[e^{-\eta\Delta_1}\right]\right\|=\left\|\int_\Omega e^{-\eta\Delta_1}\,dP\right\|\notag
\\&\leq \int_\Omega\left\|e^{-\eta\Delta_1}\right\|\,dP\leq \int_{\{\Delta_1\leq 1\}} e^{-Re(\eta)\Delta_1}\,dP+\int_{\{\Delta_1>1\}} e^{-Re(\eta)\Delta_1}\,dP\notag
\\&\leq \int_{\{\Delta_1\leq 1\}}\,dP+e^{-Re(\eta)\Delta_1}\int_{\{\Delta_1>1\}}\,dP=P\{\Delta_1\leq 1\}+e^{-Re(\eta)}\left(1-P\{\Delta_1\leq 1\}\right).
\end{align*}
Notice that $\left\|L(\eta)\right\|<1$ if and only if $e^{-Re(\eta)}<1$, which is true if and only if $Re(\eta)>0$.

Assume that $Re(\theta)\geq 0$ and $\|z\|\leq 1$. We will show that exactly one of these inequalities must be strict in order to ensure $Re(\eta)>0$ (and, thus, $|\gamma(z,\theta)\|<1$).

\underline{Case 1}: $Re(\theta)>0$, $\|z\|=1$
\begin{align}
Re(\eta)=Re(\theta)+\lambda(1-Re(g(z))).\label{BProof2}
\end{align}
Therefore, if $Re(g(z))\leq 1$ (and of course the real part of the PGF is nonnegative since for each $k$, $a_k>0$), we have $Re(\eta)>0$ as required. Since $g(z)$ is analytic and $\|g(z)\|\leq 1$ on the closed unit ball in the complex plane and $g(0)=0$, the Schwarz Lemma implies $\|g(z)\|\leq\|z\|$ in the closed unit ball, so we have
\begin{align}
0\leq Re(g(z))\leq \|g(z)\|\leq \|z\|=1.\label{BProof3}
\end{align}
Thus, the result of \eqref{BProof3} in \eqref{BProof2} implies that ${Re(\eta)=Re(\theta)+\lambda(1-Re(g(z)))\geq Re(\theta)}$. Since ${Re(\theta)>0}$ by the assumption, ${\|L(\theta+\lambda-\lambda g(z))\|<1}$.

\underline{Case 2}: $Re(\theta)=0$, $\|z\|<1$
\begin{align*}
Re(\eta) = \lambda(1-Re(g(z))).
\end{align*}
In this case, we again have ${Re(g(z))\leq\|z\|}$, but in this case we also have $\|z\|<1$, so $Re(\eta)>0$. Therefore, ${\|\gamma(z,\theta)\|=\|L(\theta+\lambda-\lambda g(z))\|<1}$.
\end{proof}

\end{appendices}
\end{document}